\documentclass[11pt,a4paper]{article}
\usepackage[pagewise]{lineno}
\usepackage{amsmath}
\usepackage{amssymb}
\usepackage{amsthm}
\usepackage{latexsym}
\usepackage{amsbsy}
\usepackage{amscd}
\usepackage{mathrsfs}
\usepackage{tipa}
\usepackage{txfonts}
\usepackage{amsfonts}
\usepackage{listings}
\usepackage{tikz}
\usetikzlibrary{patterns,snakes}
\usepackage{hyperref}
\newtheorem{thm}{Theorem}[section]
\newtheorem{cor}[thm]{Corollary}
\newtheorem{lem}[thm]{Lemma}

\newtheorem{exm}{Example}
\newtheorem{prop}[thm]{Proposition}
\newtheorem{defn}[thm]{Definition}
\newtheorem{rem}[thm]{Remark}
\newtheorem{defn-prop}[thm]{Definition-Proposition}

\newcommand{\qihao}{\fontsize{7.25pt}{\baselineskip}\selectfont}

\usepackage[all,poly]{xy}
\usepackage[all]{xy}
\usepackage{xypic}

\begin{document}

\begin{center}
{\Large \bf Cluster automorphism groups and automorphism groups of exchange graphs

\bigskip
{\large Wen Chang
\footnote{Supported by Shaanxi Province, Shaanxi Normal University and the NSF of China (Grant 11601295)} and
 Bin Zhu\footnote{Supported by the NSF of China (Grant 11671221)}}}


\end{center}

\begin{abstract}
For a coefficient free cluster algebra $\mathcal{A}$, we study the cluster automorphism group $Aut(\mathcal{A})$ and the automorphism group $Aut(E_{\mathcal{A}})$ of its exchange graph $E_{\mathcal{A}}$. We show that these two groups are isomorphic with each other, if $\mathcal{A}$ is of finite type excepting types of rank two and type $F_4$, or if $\mathcal{A}$ is of skew-symmetric finite mutation type.
\end{abstract}

\def\s{\stackrel}
\def\Longrightarrow{{\longrightarrow}}

\def\ggz{\Gamma}
\def\bz{\beta}
\def\az{\alpha}
\def\gz{\gamma}
\def\da{\delta}
\def\zz{\zeta}
\def\thz{\theta}
\def\ra{\rightarrow}

\def\FS{\mathfrak{S}}
\def\F1{\mathfrak{F}}
\def\A{\mathcal{A}}
\def\B{\mathcal{B}}
\def\C{\mathcal{C}}
\def\D{\mathcal{D}}
\def\F{\mathcal{F}}
\def\H{\mathcal{H}}
\def\I{\mathcal {I}}
\def\P{\mathbb{P}}
\def\T{\mathbb{T}}
\def\R{\mathcal{R}}
\def\Si{\Sigma}
\def\S{\Sigma}
\def\L{\mathscr{L}}
\def\l{\mathcal{l}}
\def\U{\mathscr{U}}
\def\V{\mathscr{V}}
\def\W{\mathscr{W}}
\def\X{\mathscr{X}}
\def\Y{\mathscr{Y}}
\def\M{{\bf{Mut}}}
\def\MM#1{{\bf{(CM#1)}}}
\def\t{{\tau }}
\def\b{\textbf{d}}
\def\K{{\cal K}}
\def\righta{\rightarrow}
\def\G{{\Gamma}}
\def\x{\mathbf{x}}
\def\ex{{\mathbf{ex}}}
\def\fx{{\mathbf{fx}}}
\def\Aut{\mbox{Aut}}
\def\add{\mbox{add}}
\def\coker{\mbox{coker}}
\def\End{\mbox{End}}
\def\Ext{\mbox{Ext}}
\def\Gr{\mbox{Gr}}
\def\Hom{\mbox{Hom}}
\def\id{\mbox{id}}
\def\ind{\mbox{ind}}
\def\Int{\mbox{Int}}
\def\deg{\mbox{deg}}
\def \text{\mbox}
\def\m{\multiput}
\def\mul{\multiput}
\def\c{\circ}
\def \text{\mbox}
\def\t{\tilde}

\newcommand{\homeo}{\textup{Homeo}^+(S,M)}
\newcommand{\homeoo}{\textup{Homeo}_0(S,M)}
\newcommand{\mg}{\mathcal{MG}(S,M)}
\newcommand{\mmg}{\mathcal{MG}_{\bowtie}(S,M)}
\newcommand{\Z}{\mathbb{Z}}
\newcommand{\Q}{\mathbb{Q}}
\newcommand{\N}{\mathbb{N}}

\hyphenation{ap-pro-xi-ma-tion}

\medskip


\section{Introduction}
Cluster algebras are introduced by Sergey Fomin and Andrei Zelevinsky in \cite{FZ02}. In this paper we consider cluster algebras with trivial coefficients, which can be defined through a skew-symmetrizable square matrix. Such a cluster algebra is a $\Z$-subalgebra of rational function field with $n$ indeterminates. More precisely, a {\it seed} is a pair consisting of a set ({\it cluster}) of $n$ indeterminates ({\it cluster variables}) in the field and a skew-symmetrizable square matrix ({\it exchange matrix}) of size $n$. Starting from an initial seed, we get a new seed by an operate so called mutation. 
Then the cluster algebra is algebraic-generated by all the cluster variables obtained by iterated mutations. The cluster algebra has nice combinatorial structures which are (in some sense) given by mutations, and these structures are captured by its exchange graph, which is a graph with seeds as vertices and with mutations as edges. \\

We focus in this paper on two special types of cluster algebras: the {\it finite type} and the {\it finite mutation type}. Cluster algebras of finite type are those algebras with finite number of clusters. They are classified in \cite{FZ03}, which corresponds to the Killing-Cartan classification of complex semisimple Lie algebras, equivalently, corresponds to the classification of root systems in Euclidean space. If there are finite many matrix classes in the seeds of a cluster algebra, then we say it is of finite mutation type, where two matrices are in the same class if one of them can be obtained from the other by simultaneous relabeling of the rows and columns. The cluster algebras of finite mutation type with skew-symmetric exchange matrices are classified in \cite{FST12}, a large class of them arises from marked Riemann surfaces (possibly with boundary) \cite{FST08}, and there are 11 exceptional ones. The classification of skew-symmetrizable cluster algebras of finite mutation type is given in \cite{FST11} via operations so called unfoldings upon the skew-symmetric cluster algebras of finite mutation type.\\

We consider the relations in this paper between two groups associated to the cluster algebras. One is the cluster automorphism group consisting of {\it cluster automorphisms}, which are permutations of the clusters that commutate with mutations. This group is introduced in \cite{ASS12} for a coefficient free cluster algebra, and in \cite{CZ15} for a cluster algebra with coefficients, it reveals the combinatorial and algebraic symmetries of the cluster algebra. Another is the automorphism group of the exchange graph, which consists of graph automorphism of the exchange graph. This group describes the symmetries of the exchange graph, in other words, describes combinatorial symmetries of the cluster algebra. The problem that consider the relations between these two groups is stated in \cite{S14}.\\

The exchange graph is a fairly coarse invariant of a cluster algebra, e.g. all infinite type cluster algebras of rank 2 have the same exchange graph. This article suggests that, nonetheless, the exchange graph is already rich enough to capture most of the symmetries of the cluster algebra.\\

For a coefficient free cluster algebra $\A$ with exchange graph $E_\A$, we write the cluster automorphism group of $\A$ and the automorphism group of $E_\A$ as $Aut(\A)$ and $Aut(E_\A)$ respectively. In general, $Aut(\A)$ is a subgroup of $Aut(E_\A)$, and may be a proper subgroup, see Example \ref{exm:rank 2 case} \ref{exm:infinite rank 2 case2}. The main result of this paper is that these two groups are isomorphic with each other, if $\mathcal{A}$ is of finite type, excepting types of rank two and type $F_4$ (Theorem \ref{thm:auto gp of finite type ex graph}), or $\mathcal{A}$ is of skew-symmetric finite mutation type (Theorem \ref{thm:auto gp of finite mutation type ex graph}). Therefore in some degree, for these cluster algebras, the algebraic symmetries are also captured by the exchange graphs. In particular, we compute the automorphism group of the exchange graph of a finite type cluster algebra in table \ref{table:auto group of exchange graphs}, see Remark \ref{rem:auto-gp-of-finite-type}.
\begin{table}[ht]
\begin{equation*}
\begin{array}{cc}
\textrm{Dynkin type} & \textrm{Automorphism group $Aut(E_\A)$} \\
\hline
A_{n} (n\geqslant 2) & \mathbb{D}_{n+3}\\
B_{2} &\mathbb{D}_{6}\\
B_{n} (n\geqslant 3) & \mathbb{D}{n+1}\\
C_{2} & \mathbb{D}_{6}\\
C_n (n\geqslant 3) & \mathbb{D}_{n+1}\\
D_4 & \mathbb{D}_4\times S_3 \\
D_n (n\geqslant 5) & \mathbb{Z}_2\\
E_6 & \mathbb{D}_{14}\\
E_7 & \mathbb{D}_{10}\\
E_8 & \mathbb{D}_{16}\\
F_{4} & \mathbb{D}_7\rtimes \Z_2\\
G_{2} & \mathbb{D}_{8}\\
\end{array}
\end{equation*}
\smallskip
\caption{Automorphism groups of exchange graphs of cluster algebras of finite type}
\label{table:auto group of exchange graphs}
\end{table}

To prove these results, we describe $E_\A$ more precisely. In Section \ref{Layers of geodesic loops}, we define layers of geodesic loops of $E_\A$ by using the distance of a vertex to a fixed vertex on $E_\A$. An easy observation is that an isomorphism of exchange graphs should maintain the combinatorial numbers of the layers of geodesic loops based on the corresponding vertices see Remark \ref{rem: layers of geodesic loops}(4). By this observation, we directly show in Examples \ref{exm:A3 type exchange graph}, \ref{exm:B3 C3 type exchange graph}, \ref{exm:infinite rank 3 case}, \ref{exm:infinite rank 3 case 2} that for a cluster algebra of type $A_3$, $B_3$, $C_3$, $\tilde{A}_2$ or $T_3$ (the cluster algebra from an once punctured torus), we have $Aut(\A)\cong Aut(E_\A)$. For the general cases we reduce them to above five cases (Theorems \ref{thm:auto gp of finite type ex graph}, \ref{thm:auto gp of finite mutation type ex graph}).\\

The paper is organized as follows: we recall preliminaries on cluster algebras, cluster algebras of finite mutation type and cluster automorphisms in section \ref{Preleminaries}, then we prove the main Theorems in section \ref{Sec exchange graph}.

\section{Preliminaries}
\label{Preleminaries}

\subsection{Cluster algebras}
\begin{defn} \cite{FZ02}(Labeled seeds).\label{def: labeled seeds}
A \emph{labeled seed} is a pair $\Sigma=(\x, B)$, where
\begin{itemize}
\item  $\x=\{x_{1},x_{2},\cdot \cdot \cdot ,x_{n}\}$ is an ordered set of $n$ indeterminates;
\item  $B=(b_{x_jx_i})_{n\times n}\in M_{n\times n}(\Z)$ is a skew-symmetrizable matrix labeled by $\x\times \x$, that is, there exists a diagonal matrix $D$ with positive integer entries such that $DB$ is skew-symmetric.
\end{itemize}
\end{defn}
The set $\x$ is called the \emph{cluster} with elements the \emph{cluster variables}, and $B$ is called the \emph{exchange matrix}. An element $b_{x_jx_i}$ in $B$ is also written as $b_{ji}$ for brevity. We always assume through the paper that $B$ is indecomposable, that is, for any $1\leqslant i,j\leqslant n$, there is a sequence $i_0=i,i_1,\cdot \cdot \cdot, i_m, i_{m+1}=j$, such that $b_{i_{k},i_{k+1}}\neq 0$ for any $0\leqslant k\leqslant m$. We also assume that $n>1$ for convenience.
One may produce a new labeled seed by a mutation at direction $k$ for any cluster variable $x_k$.

\begin{defn} \cite{FZ02}(Seed mutations).\label{def: mutation}
The labeled seed $\mu_k(\S)=(\mu_k(\x),\mu_k(B))$ obtained by the \emph{mutation} of $\S$ in the direction $k$ is given by:
\begin{itemize}
				\item $\mu_k(\x) = (\x \setminus \{x_k\}) \sqcup \{\mu_{x_k,\x}(x_k)\}$ where
				$$x_k\mu_{x_k,\x}(x_k) = \prod_{\substack{1\leqslant j\leqslant n~; \\ b_{jk}>0}} {x_j}^{b_{jk}} + \prod_{\substack{1\leqslant j\leqslant n~; \\ b_{jk}<0}} {x_j}^{-b_{jk}}.$$
				\item $\mu_k(B)=(b'_{ji})_{n\times n} \in M_{n\times n}(\Z)$ is given by
					$$b'_{ji} = \left\{\begin{array}{ll}
						- b_{ji} & \textrm{ if } i=k \textrm{ or } j=k~; \\
						b_{ji} + \frac 12 (|b_{ji}|b_{ik} + b_{ji}|b_{ik}|) & \textrm{ otherwise.}
					\end{array}\right.$$
\end{itemize}
\end{defn}
It is easy to check that a mutation is an involution, that is $\mu_k\mu_k(\S)=\S$.

\begin{defn} \cite{FZ07}(n-cluster patterns).\label{def:n-regular patterns}
An n-regular tree $\T_n$ is a diagram, whose edges are labeled by $1,2,\cdots,n$, such that the $n$ edges emanating from each vertex receive different labels.
A \emph{n-cluster pattern} is an assignment
of a labeled seed $\Sigma_t=(\x_t, B_t)$ to every vertex $t \in \T_n$, so that the labeled seeds assigned to the
endpoints of any edge labeled by $k$ are obtained from each
other by the seed mutation in direction~$k$.
The elements of $\Sigma_t$ are written as follows:
\begin{equation}
\label{eq:seed-labeling}
\x_t = (x_{1;t}\,,\dots,x_{n;t})\,,\quad
B_t = (b^t_{ij})\,.
\end{equation}
\end{defn}
Note that $\T_n$ is in fact determined by any fixed labeled seed on it.
Now we are ready to define cluster algebras.
\begin{defn} \cite{FZ07}(Cluster algebras).\label{def: cluster algebras}
Given a seed $\S$ and a cluster pattern $\T_n$ associated to it, we denote
\begin{equation}
\label{eq:cluster-variables}
\X
= \bigcup_{t \in \T_n} \x_t
= \{ x_{i,t}\,:\, t \in \T_n\,,\ 1\leq i\leq n \} \ ,
\end{equation}
the union of clusters of all the seeds in the pattern.
We call the elements $x_{i,t}\in \X$ the \emph{cluster variables}.
The \emph{cluster algebra} $\A$ associated with $\S$ is the $\Z$-subalgebra of the rational function field $\F=\Q(x_{1},x_{2},\cdot \cdot \cdot ,x_{n})$,
generated by all cluster variables: $\A = \Z[\X]$.
\end{defn}

For a skew-symmetrizable matrix $B=(b_{ji})_{n\times n}$, one can associate it to a \emph{valued quiver} (quiver for brevity) $Q=(Q_0,Q_1,\upsilon)$ as follows: $Q_0=\{1,2,\cdot\cdot\cdot,n\}$ is a set of vertices; for any two vertices $j$ and $i$, if $b_{ji} > 0$, then there is an arrow $\alpha$ from $j$ to $i$, these arrows form the set $Q_1$; for an arrow $\alpha$ from $j$ to $i$, we assign it a pair of values: $(\upsilon_1(\alpha),\upsilon_2(\alpha))=(b_{ji},-b_{ij})$. Since $B$ is an indecomposable skew-symmetrizable matrix, the defined valued quiver $Q$ is connected and there is no loops nor $2$-cycles in $Q$. Then we can define a mutation of the valued quiver by the mutation of the matrix, we refer to \cite{FZ02,K12} for details. We say two quivers $Q$ and $Q'$ \emph{mutation equivalent}, if the corresponding matrices are mutation equivalent, that is, one of them can be obtained from the other one by a finite sequence of mutations. We also write $(\x ,Q)$ for the labeled seed $(\x, B)$, and write $\A_Q$ to the cluster algebra defined by $\Sigma$. The quiver and the defined cluster algebra are called skew-symmetric, if the corresponding matrix is skew-symmetric. If the cluster algebra is of finite type \cite{FZ03} or of skew-symmetric type, then the cluster determines the quiver \cite{GSV08}, and we denote the quiver of a cluster $\x$ by $Q(\x)$.

\begin{exm}\label{exm:quiver and matrix}
Let $B$ be the following matrix, it is a skew-symmetrizable matrix with skew-symmetrizer $D=diag\{2,2,1,1\}$. The quiver corresponding to $B$ is $Q$, where we always delete the trivial pairs of values $(1,1)$, and replace a arrow assigning pair $(m,m)$ by $m$ arrows.
\begin{center}
{\begin{tikzpicture}
\node[] (C) at (-3,-3)  {$B$
$~=~$
                        $\left(
                          \begin{array}{cccc}
                            0 & 1 & 0 & 0\\
                            -1 & 0 & -1 & 0\\
                            0 & 2 & 0 & 2\\
                            0 & 0 & -2 & 0\\
                          \end{array}
                        \right)$};
\end{tikzpicture}}
\end{center}

\begin{center}
{\begin{tikzpicture}
\node[] (C) at (-2.5,0)  {$Q~:$};
\node[] (C) at (-1.5,0)  {$1$};
\node[] (C) at (0,0)  {$2$};	
\node[] (C) at (1.5,0)  {$3$};
\node[] (C) at (3,0)  {$4$};
\node[] (C) at (0.75,0.3)  {\qihao{(2,1)}};
\draw[<-,thick] (-0.2,0) -- (-1.3,0);
\draw[<-,thick] (0.2,0) -- (1.3,0);
\draw[->,thick] (1.7,0.05) -- (2.8,0.05);
\draw[->,thick] (1.7,-0.05) -- (2.8,-0.05);
\end{tikzpicture}}
\end{center}
\end{exm}

\begin{defn} \cite{FZ07}(Seeds)
\label{def:seeds}
Given two labeled seeds $\Sigma=(\x, B)$
and $\Sigma'=(\x', B')$, we say that they
define the same \emph{seed}
if $\Sigma'$ is obtained from $\Sigma$ by simultaneous relabeling of the
sets $\x$ and the corresponding relabeling of
the rows and columns of $B$.
\end{defn}
We denote by $[\S]$ the seed represented by a labeled seed $\S$. The cluster $\x$ of a seed $[\S]$ is a unordered $n$-element set. For any $x\in\x$, there is a well-defined mutation $\mu_{x}([\S])=[\mu_{k}(\S)]$ of $[\S]$ at direction $x$, where $x=x_k$. For two same rank skew-symmetrizable matrices $B$ and $B'$, we say $B\cong B'$, if $B'$ is obtained from $B$ by simultaneous relabeling of the rows and columns of $B$. Then the exchange matrices in any two labeled seeds representing a same seed are isomorphic. The isomorphism of two exchange matrices induces an isomorphism of corresponding quivers. For convenience, in the rest of the paper, we also denote by $\S$ the seed $[\S]$ represented by $\S$.
\begin{defn} \cite{FZ07}(Exchange graphs)
\label{def:exchange-graph}
The \emph{exchange graph} of a cluster algebra is the $n$-regular graph whose vertices are the seeds of the cluster algebra and whose
edges connect the seeds related by a single mutation. We denote by $E_\A$ the exchange graph of a cluster algebra $\A$.
\end{defn}
Clearly, the exchange graph of a cluster algebra is a quotient graph of the $n$-regular tree, its vertices are equivalent classes of labeled seeds.
The exchange graph not necessary be a finite graph, if it is finite, then we say the corresponding cluster algebra (and its cluster pattern) is of \emph{finite type}.
\begin{defn} \cite[page 70]{FZ03}(Cluster complexes)
\label{def:cluster-cmoplex}
A cluster complex $\Delta$ of $\A$ is a simplicial complex on the ground set $\X$ with the clusters as the maximal simplices.
\end{defn}
Then $\Delta$ is an $n$-dimensional complex. In particular, if $\A$ is of finite type or skew-symmetric, then the vertices of $E_\A$ are clusters, thus the dual graph of $\Delta$ is $E_\A$.
\subsection{Finite types and finite mutation types}\label{sec:finite types and finite mutation types}
By the classification of cluster algebras of finite type \cite{FZ03}, a cluster algebra is of finite type if and only if there is a seed whose quiver is one of quivers depicted in Figure \ref{fig:quivers of finite types}.
Note that the underlying graphs of quivers in Figure \ref{fig:quivers of finite types} are trees, thus any two quivers with the same underlying graph are mutation-equivalent.\\
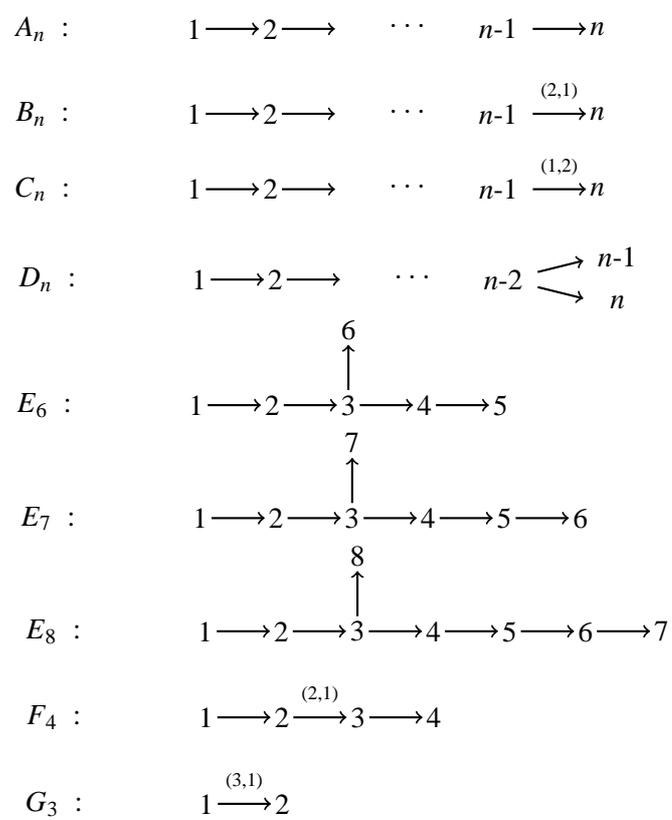
\begin{figure}
\begin{center}
{\begin{tikzpicture}

\node[] (C) at (0,0)  {\begin{tikzpicture}
\node[] (C) at (0,0)  {$A_n~:$};
\node[] (C) at (2,0)  {$1$};
\node[] (C) at (3,0)  {$2$};
\node[] (C) at (4.8,0)  {$\cdot\cdot\cdot$};
\node[] (C) at (6,0)  {$n$-$1$};
\node[] (C) at (7.3,0)  {$n$};
\draw[->,thick] (2.15,0) -- (2.85,0);
\draw[->,thick] (3.15,0) -- (3.85,0);
\draw[->,thick] (6.45,0) -- (7.15,0);
                        \end{tikzpicture}};

\node[] (C) at (0,-1)  {\begin{tikzpicture}
\node[] (C) at (0,0)  {$B_n~:$};
\node[] (C) at (2,0)  {$1$};
\node[] (C) at (3,0)  {$2$};
\node[] (C) at (4.8,0)  {$\cdot\cdot\cdot$};
\node[] (C) at (6,0)  {$n$-$1$};
\node[] (C) at (7.3,0)  {$n$};
\node[] (C) at (6.8,0.3)  {\qihao{(2,1)}};
\draw[->,thick] (2.15,0) -- (2.85,0);
\draw[->,thick] (3.15,0) -- (3.85,0);
\draw[->,thick] (6.45,0) -- (7.15,0);
                        \end{tikzpicture}};

\node[] (C) at (0,-2)  {\begin{tikzpicture}
\node[] (C) at (0,0)  {$C_n~:$};
\node[] (C) at (2,0)  {$1$};
\node[] (C) at (3,0)  {$2$};
\node[] (C) at (4.8,0)  {$\cdot\cdot\cdot$};
\node[] (C) at (6,0)  {$n$-$1$};
\node[] (C) at (7.3,0)  {$n$};
\node[] (C) at (6.8,0.3)  {\qihao{(1,2)}};
\draw[->,thick] (2.15,0) -- (2.85,0);
\draw[->,thick] (3.15,0) -- (3.85,0);
\draw[->,thick] (6.45,0) -- (7.15,0);
                        \end{tikzpicture}};

\node[] (C) at (0,-3.3)  {\begin{tikzpicture}
\node[] (C) at (-0.8,0)  {~~};
\node[] (C) at (0,0)  {$D_n~:$};
\node[] (C) at (2,0)  {$1$};
\node[] (C) at (3,0)  {$2$};
\node[] (C) at (4.8,0)  {$\cdot\cdot\cdot$};
\node[] (C) at (6,0)  {$n$-$2$};
\node[] (C) at (7.5,0.3)  {$n$-$1$};
\node[] (C) at (7.5,-0.3)  {$n$};
\draw[->,thick] (2.15,0) -- (2.85,0);
\draw[->,thick] (3.15,0) -- (3.85,0);
\draw[->,thick] (6.45,0.1) -- (7.05,0.25);
\draw[->,thick] (6.45,-0.1) -- (7.05,-0.25);
                        \end{tikzpicture}};

\node[] (C) at (0,-4.5)  {\begin{tikzpicture}
\node[] (C) at (0,0)  {$E_6~:$};
\node[] (C) at (2,0)  {$1$};
\node[] (C) at (3,0)  {$2$};
\node[] (C) at (4,0)  {$3$};
\node[] (C) at (5,0)  {$4$};
\node[] (C) at (6,0)  {$5$};
\node[] (C) at (4,1)  {$6$};
\draw[->,thick] (2.15,0) -- (2.85,0);
\draw[->,thick] (3.15,0) -- (3.85,0);
\draw[->,thick] (4.15,0) -- (4.85,0);
\draw[->,thick] (5.15,0) -- (5.85,0);
\draw[->,thick] (4,0.2) -- (4,0.8);
\node[] (C) at (7.3,0)  {~~};
                        \end{tikzpicture}};

\node[] (C) at (0,-6)  {\begin{tikzpicture}
\node[] (C) at (0,0)  {$E_7~:$};
\node[] (C) at (2,0)  {$1$};
\node[] (C) at (3,0)  {$2$};
\node[] (C) at (4,0)  {$3$};
\node[] (C) at (5,0)  {$4$};
\node[] (C) at (6,0)  {$5$};
\node[] (C) at (7,0)  {$6$};
\node[] (C) at (4,1)  {$7$};
\draw[->,thick] (2.15,0) -- (2.85,0);
\draw[->,thick] (3.15,0) -- (3.85,0);
\draw[->,thick] (4.15,0) -- (4.85,0);
\draw[->,thick] (5.15,0) -- (5.85,0);
\draw[->,thick] (6.15,0) -- (6.85,0);
\draw[->,thick] (4,0.2) -- (4,0.8);
\node[] (C) at (7.2,0)  {~~};
                        \end{tikzpicture}};

\node[] (C) at (0.4,-7.5)  {\begin{tikzpicture}
\node[] (C) at (0,0)  {$E_8~:$};
\node[] (C) at (2,0)  {$1$};
\node[] (C) at (3,0)  {$2$};
\node[] (C) at (4,0)  {$3$};
\node[] (C) at (5,0)  {$4$};
\node[] (C) at (6,0)  {$5$};
\node[] (C) at (7,0)  {$6$};
\node[] (C) at (8,0)  {$7$};
\node[] (C) at (4,1)  {$8$};
\draw[->,thick] (2.15,0) -- (2.85,0);
\draw[->,thick] (3.15,0) -- (3.85,0);
\draw[->,thick] (4.15,0) -- (4.85,0);
\draw[->,thick] (5.15,0) -- (5.85,0);
\draw[->,thick] (6.15,0) -- (6.85,0);
\draw[->,thick] (7.15,0) -- (7.85,0);
\draw[->,thick] (4,0.2) -- (4,0.8);
\node[] (C) at (-0.5,0)  {~~};
                        \end{tikzpicture}};

\node[] (C) at (-1,-9)  {\begin{tikzpicture}
\node[] (C) at (0,0)  {$F_4~:$};
\node[] (C) at (2,0)  {$1$};
\node[] (C) at (3,0)  {$2$};
\node[] (C) at (4,0)  {$3$};
\node[] (C) at (5,0)  {$4$};
\node[] (C) at (3.5,0.3)  {\qihao{(2,1)}};
\draw[->,thick] (2.15,0) -- (2.85,0);
\draw[->,thick] (3.15,0) -- (3.85,0);
\draw[->,thick] (4.15,0) -- (4.85,0);
                        \end{tikzpicture}};

\node[] (C) at (-2,-10.15)  {\begin{tikzpicture}
\node[] (C) at (0,0)  {$G_3~:$};
\node[] (C) at (2,0)  {$1$};
\node[] (C) at (3,0)  {$2$};
\node[] (C) at (2.5,0.3)  {\qihao{(3,1)}};
\draw[->,thick] (2.15,0) -- (2.85,0);
                        \end{tikzpicture}};

\end{tikzpicture}}
\caption{Quivers of finite type}
\label{fig:quivers of finite types}
\end{center}
\end{figure}


\begin{defn} \cite{FST08,FST12}\label{defblock}
A \emph{block} is a quiver isomorphic to one of the quivers with black / white colored vertices shown on Figure \ref{fig:blocks-I-V}. Vertices marked in white are called \emph{outlets}. A connected quiver $Q$ is called \emph{block-decomposable} (\emph{decomposable} for brevity) if it can be obtained from a collection of blocks by identifying outlets of different blocks along some partial matching (matching of outlets of the same block is not allowed), where two arrows with same endpoints and opposite directions cancel out. If $Q$ is not block-decomposable then we call $Q$ \emph{non-decomposable}.

\begin{figure}
\begin{center}
{\begin{tikzpicture}
\node[] (C) at (-6,0)  {\begin{tikzpicture}
\node[] (C) at (-0.5,0)  {$\circ$};
\node[] (C) at (0.5,0)  {$\circ$};
\draw[->,thick] (-0.44,0) -- (0.44,0);
                        \end{tikzpicture}};

\node[] (C) at (-4,0)  {\begin{tikzpicture}
\node[] (C) at (-0.5,0)  {$\circ$};
\node[] (C) at (0.5,0)  {$\circ$};
\node[] (C) at (0,0.74)  {$\circ$};
\draw[<-,thick] (-0.44,0) -- (0.44,0);
\draw[->,thick] (-0.46,0.04) -- (-0.04,0.7);
\draw[<-,thick] (0.46,0.04) -- (0.04,0.7);
                        \end{tikzpicture}};

\node[] (C) at (-2,0)  {\begin{tikzpicture}
\node[] (C) at (-0.5,0)  {$\bullet$};
\node[] (C) at (0.5,0)  {$\bullet$};
\node[] (C) at (0,0.74)  {$\circ$};
\draw[->,thick] (-0.46,0.04) -- (-0.04,0.7);
\draw[->,thick] (0.46,0.04) -- (0.04,0.7);
                        \end{tikzpicture}};

\node[] (C) at (0,0)  {\begin{tikzpicture}
\node[] (C) at (-0.5,0)  {$\bullet$};
\node[] (C) at (0.5,0)  {$\bullet$};
\node[] (C) at (0,0.74)  {$\circ$};
\draw[<-,thick] (-0.46,0.04) -- (-0.04,0.7);
\draw[<-,thick] (0.46,0.04) -- (0.04,0.7);
                        \end{tikzpicture}};

\node[] (C) at (2,0)  {\begin{tikzpicture}
\node[] (C) at (-0.5,0)  {$\circ$};
\node[] (C) at (0.5,0)  {$\circ$};
\node[] (C) at (0,0.74)  {$\bullet$};
\node[] (C) at (0,-0.74)  {$\bullet$};
\draw[<-,thick] (-0.44,0) -- (0.44,0);
\draw[->,thick] (-0.46,0.04) -- (-0.04,0.7);
\draw[<-,thick] (0.46,0.04) -- (0.04,0.7);
\draw[->,thick] (-0.46,-0.04) -- (-0.04,-0.7);
\draw[<-,thick] (0.46,-0.04) -- (0.04,-0.7);
                        \end{tikzpicture}};

\node[] (C) at (4,0)  {\begin{tikzpicture}
\node[] (C) at (-0.5,0)  {$\bullet$};
\node[] (C) at (0.5,0)  {$\bullet$};
\node[] (C) at (-0.5,1)  {$\bullet$};
\node[] (C) at (0.5,1)  {$\bullet$};
\node[] (C) at (0,0.5)  {$\circ$};
\draw[->,thick] (-0.44,0) -- (0.44,0);
\draw[<-,thick] (-0.46,0.04) -- (-0.04,0.46);
\draw[->,thick] (0.46,0.04) -- (0.04,0.46);
\draw[<-,thick] (-0.44,1) -- (0.44,1);
\draw[->,thick] (-0.46,0.96) -- (-0.04,0.54);
\draw[<-,thick] (0.46,0.96) -- (0.04,0.54);
\draw[->,thick] (-0.5,0) -- (-0.5,0.94);
\draw[<-,thick] (0.5,0.06) -- (0.5,0.94);
                        \end{tikzpicture}};

\node[] (C) at (-6,-1.3)  {I};
\node[] (C) at (-4,-1.3)  {II};
\node[] (C) at (-2,-1.3)  {IIIa};
\node[] (C) at (0,-1.3)  {IIIb};
\node[] (C) at (2,-1.3)  {IV};
\node[] (C) at (4,-1.3)  {V};
\end{tikzpicture}}
\caption{Blocks. Outlets are colored white, dead ends are black.}
\label{fig:blocks-I-V}
\end{center}
\end{figure}
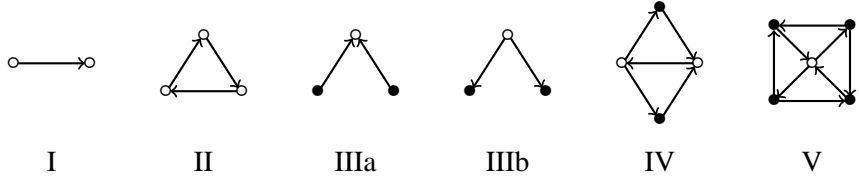
\end{defn}

Then it is proved in \cite[Theorem 13.3]{FST08} that a quiver is decomposable if and only if it is a quiver of a triangulation of an oriented marked Riemann surface, and thus a quiver mutation equivalent to a decomposable quiver is also decomposable. Note that all arrow multiplicities of a decomposable quiver are $1$ or $2$. Therefore decomposable quivers are mutation finite. It is clear that a quiver of rank two, that is, a quiver with two vertices, is mutation finite. Besides these two kinds of quivers, there are exactly $11$ exceptional skew-symmetric quivers of finite mutation type, see Theorem 6.1 in \cite{FST12}. We list the exceptional quivers in Figure \ref{fig:exceptional mutation finite types}.

\begin{figure}
\begin{center}
{\begin{tikzpicture}

\node[] (C) at (-4.4,1)  {$E_6~$};
\node[] (C) at (-3.4,1)  {$E_7~$};
\node[] (C) at (-2.4,1)  {$E_8~$};
\node[] (C) at (0,0)  {\begin{tikzpicture}
\node[] (C) at (0,0)  {$\tilde{E}_6~:$};
\node[] (C) at (2,0)  {$1$};
\node[] (C) at (3,0)  {$2$};
\node[] (C) at (4,0)  {$3$};
\node[] (C) at (5,0)  {$4$};
\node[] (C) at (6,0)  {$5$};
\node[] (C) at (4,1)  {$6$};
\node[] (C) at (4,2)  {$7$};
\draw[->,thick] (2.15,0) -- (2.85,0);
\draw[->,thick] (3.15,0) -- (3.85,0);
\draw[->,thick] (4.15,0) -- (4.85,0);
\draw[->,thick] (5.15,0) -- (5.85,0);
\draw[->,thick] (4,0.2) -- (4,0.8);
\draw[->,thick] (4,1.2) -- (4,1.8);
\node[] (C) at (9,0)  {~~};
                        \end{tikzpicture}};

\node[] (C) at (-0.5,-2)  {\begin{tikzpicture}
\node[] (C) at (0,0)  {$\tilde{E}_7~:$};
\node[] (C) at (2,0)  {$1$};
\node[] (C) at (3,0)  {$2$};
\node[] (C) at (4,0)  {$3$};
\node[] (C) at (5,0)  {$4$};
\node[] (C) at (6,0)  {$5$};
\node[] (C) at (7,0)  {$6$};
\node[] (C) at (8,0)  {$7$};
\node[] (C) at (5,1)  {$8$};
\draw[->,thick] (2.15,0) -- (2.85,0);
\draw[->,thick] (3.15,0) -- (3.85,0);
\draw[->,thick] (4.15,0) -- (4.85,0);
\draw[->,thick] (5.15,0) -- (5.85,0);
\draw[->,thick] (6.15,0) -- (6.85,0);
\draw[->,thick] (7.15,0) -- (7.85,0);
\draw[->,thick] (5,0.2) -- (5,0.8);
                        \end{tikzpicture}};

\node[] (C) at (0,-3.5)  {\begin{tikzpicture}
\node[] (C) at (0,0)  {$\tilde{E}_8~:$};
\node[] (C) at (2,0)  {$1$};
\node[] (C) at (3,0)  {$2$};
\node[] (C) at (4,0)  {$3$};
\node[] (C) at (5,0)  {$4$};
\node[] (C) at (6,0)  {$5$};
\node[] (C) at (7,0)  {$6$};
\node[] (C) at (8,0)  {$7$};
\node[] (C) at (9,0)  {$8$};
\node[] (C) at (4,1)  {$9$};
\draw[->,thick] (2.15,0) -- (2.85,0);
\draw[->,thick] (3.15,0) -- (3.85,0);
\draw[->,thick] (4.15,0) -- (4.85,0);
\draw[->,thick] (5.15,0) -- (5.85,0);
\draw[->,thick] (6.15,0) -- (6.85,0);
\draw[->,thick] (7.15,0) -- (7.85,0);
\draw[->,thick] (8.15,0) -- (8.85,0);
\draw[->,thick] (4,0.2) -- (4,0.8);
                        \end{tikzpicture}};

\node[] (C) at (-0.45,-5.5)  {\begin{tikzpicture}
\node[] (C) at (0,0)  {$E^{(1,1)}_6~:$};
\node[] (C) at (2,0)  {$1$};
\node[] (C) at (3,0)  {$2$};
\node[] (C) at (5,0)  {$3$};
\node[] (C) at (6,0)  {$4$};
\node[] (C) at (7,0)  {$5$};
\node[] (C) at (8,0)  {$6$};
\node[] (C) at (4,0.7)  {$7$};
\node[] (C) at (4,-0.7)  {$8$};
\draw[->,thick] (2.15,0) -- (2.85,0);
\draw[->,thick] (5.15,0) -- (5.85,0);
\draw[->,thick] (7.15,0) -- (7.85,0);
\draw[->,thick] (3.95,-0.5) -- (3.95,0.5);
\draw[->,thick] (4.05,-0.5) -- (4.05,0.5);
\draw[->,thick] (3.9,0.6) -- (3.1,0.1);
\draw[<-,thick] (3.9,-0.6) -- (3.1,-0.1);
\draw[->,thick] (4.1,0.6) -- (4.9,0.1);
\draw[<-,thick] (4.1,-0.6) -- (4.9,-0.1);
\draw[->,thick] (4.2,0.65) -- (6.85,0.1);
\draw[<-,thick] (4.2,-0.65) -- (6.85,-0.1);
                        \end{tikzpicture}};

\node[] (C) at (0,-7.8)  {\begin{tikzpicture}
\node[] (C) at (-1,0)  {$E^{(1,1)}_7~:$};
\node[] (C) at (1,0)  {$1$};
\node[] (C) at (2,0)  {$2$};
\node[] (C) at (3,0)  {$3$};
\node[] (C) at (5,0)  {$4$};
\node[] (C) at (6,0)  {$5$};
\node[] (C) at (7,0)  {$6$};
\node[] (C) at (8,0)  {$7$};
\node[] (C) at (4,0.7)  {$8$};
\node[] (C) at (4,-0.7)  {$9$};
\draw[->,thick] (1.15,0) -- (1.85,0);
\draw[->,thick] (2.15,0) -- (2.85,0);
\draw[->,thick] (6.15,0) -- (6.85,0);
\draw[->,thick] (7.15,0) -- (7.85,0);
\draw[->,thick] (3.95,-0.5) -- (3.95,0.5);
\draw[->,thick] (4.05,-0.5) -- (4.05,0.5);
\draw[->,thick] (3.9,0.6) -- (3.1,0.1);
\draw[<-,thick] (3.9,-0.6) -- (3.1,-0.1);
\draw[->,thick] (4.1,0.6) -- (4.9,0.1);
\draw[<-,thick] (4.1,-0.6) -- (4.9,-0.1);
\draw[->,thick] (4.2,0.65) -- (5.85,0.1);
\draw[<-,thick] (4.2,-0.65) -- (5.85,-0.1);
                        \end{tikzpicture}};

\node[] (C) at (0.4,-10)  {\begin{tikzpicture}
\node[] (C) at (0,0)  {$E^{(1,1)}_8~:$};
\node[] (C) at (2,0)  {$1$};
\node[] (C) at (3,0)  {$2$};
\node[] (C) at (5,0)  {$3$};
\node[] (C) at (6,0)  {$4$};
\node[] (C) at (7,0)  {$5$};
\node[] (C) at (8,0)  {$6$};
\node[] (C) at (9,0)  {$7$};
\node[] (C) at (10,0)  {$8$};
\node[] (C) at (4,0.7)  {$9$};
\node[] (C) at (4,-0.75)  {$10$};
\draw[->,thick] (2.15,0) -- (2.85,0);
\draw[->,thick] (6.15,0) -- (6.85,0);
\draw[->,thick] (7.15,0) -- (7.85,0);
\draw[->,thick] (8.15,0) -- (8.85,0);
\draw[->,thick] (9.15,0) -- (9.85,0);
\draw[->,thick] (3.95,-0.5) -- (3.95,0.5);
\draw[->,thick] (4.05,-0.5) -- (4.05,0.5);
\draw[->,thick] (3.9,0.6) -- (3.1,0.1);
\draw[<-,thick] (3.9,-0.6) -- (3.1,-0.1);
\draw[->,thick] (4.1,0.6) -- (4.9,0.1);
\draw[<-,thick] (4.1,-0.6) -- (4.9,-0.1);
\draw[->,thick] (4.2,0.65) -- (5.85,0.1);
\draw[<-,thick] (4.2,-0.65) -- (5.85,-0.1);
                        \end{tikzpicture}};

\node[] (C) at (-2.5,-12.5)  {\begin{tikzpicture}
\node[] (C) at (-3,0)  {$X_6~:$};
\node[] (C) at (0,0)  {$3$};
\node[] (C) at (-1,0)  {$1$};
\node[] (C) at (-0.5,0.75)  {$2$};
\draw[<-,thick] (-0.9,0) -- (-0.1,0);
\draw[->,thick] (-0.95,0.15) -- (-0.65,0.65);
\draw[->,thick] (-0.9,0.1) -- (-0.6,0.6);
\draw[<-,thick] (-0.1,0.1) -- (-0.4,0.64);

\node[] (C) at (1,0)  {$5$};
\node[] (C) at (0.5,0.75)  {$4$};
\draw[->,thick] (0.9,0) -- (0.1,0);
\draw[<-,thick] (0.95,0.15) -- (0.65,0.65);
\draw[<-,thick] (0.9,0.1) -- (0.6,0.6);
\draw[->,thick] (0.1,0.1) -- (0.4,0.64);

\node[] (C) at (0,-1)  {$6$};
\draw[->,thick] (0,-0.85) -- (0,-0.15);
                        \end{tikzpicture}};

\node[] (C) at (-2.5,-15)  {\begin{tikzpicture}
\node[] (C) at (-3,0)  {$X_7~:$};
\node[] (C) at (0,0)  {$3$};
\node[] (C) at (-1,0)  {$1$};
\node[] (C) at (-0.5,0.75)  {$2$};
\draw[<-,thick] (-0.9,0) -- (-0.1,0);
\draw[->,thick] (-0.95,0.15) -- (-0.65,0.65);
\draw[->,thick] (-0.9,0.1) -- (-0.6,0.6);
\draw[<-,thick] (-0.1,0.1) -- (-0.4,0.64);

\node[] (C) at (1,0)  {$5$};
\node[] (C) at (0.5,0.75)  {$4$};
\draw[->,thick] (0.9,0) -- (0.1,0);
\draw[<-,thick] (0.95,0.15) -- (0.65,0.65);
\draw[<-,thick] (0.9,0.1) -- (0.6,0.6);
\draw[->,thick] (0.1,0.1) -- (0.4,0.64);

\node[] (C) at (-0.5,-0.75)  {$6$};
\node[] (C) at (0.5,-0.75)  {$7$};
\draw[<-,thick] (-0.1,-0.1) -- (-0.4,-0.64);
\draw[->,thick] (0.1,-0.1) -- (0.4,-0.64);
\draw[<-,thick] (-0.4,-0.71) -- (0.4,-0.71);
\draw[<-,thick] (-0.4,-0.79) -- (0.4,-0.79);
                        \end{tikzpicture}};

\end{tikzpicture}}
\caption{Representatives of non-decomposable quivers of finite mutation type}
\label{fig:exceptional mutation finite types}
\end{center}
\end{figure}
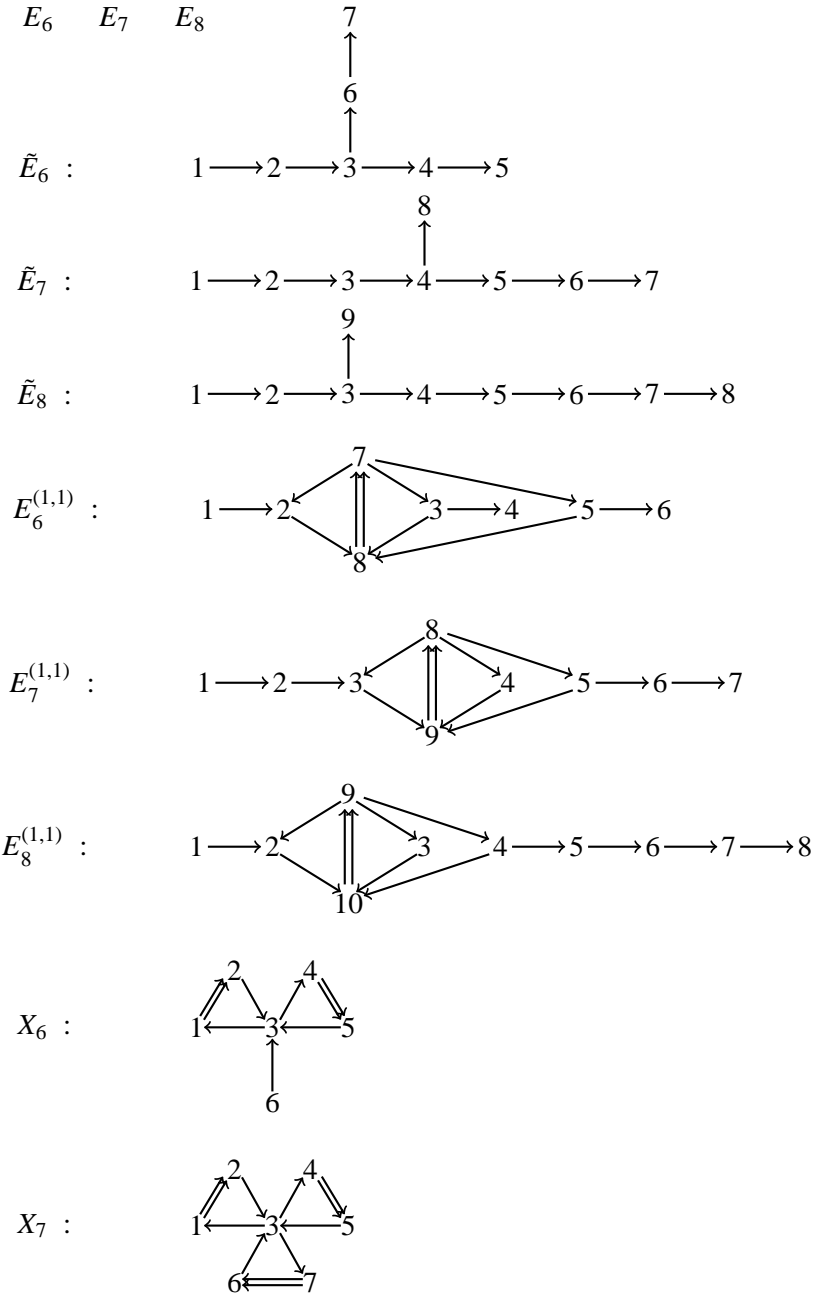
\subsection{Automorphism groups}
In this section, we recall the cluster automorphism group \cite{ASS12} of a cluster algebra, and the automorphism group of the corresponding exchange graph \cite{CZ15}.
\begin{defn} \cite{ASS12}\emph(Cluster automorphisms)\label{cluster automorphisms}
For a cluster algebra $\A$ and a $\Z$-algebra automorphism $f:\A\to\A$, we call $f$ a \emph{cluster automorphism}, if there exists a labeled seed $(\x,B)$ of $\A$ such that the following conditions are satisfied:
\begin{enumerate}
\item $f(\x)$ is a cluster;
\item $f$ is compatible with mutations, that is, for every $x\in \x$ and $y\in \x$, we have
$$f(\mu_{x,\x}(y))=\mu_{f(x),f(\x)}(f(y)).$$
\end{enumerate}
\end{defn}
Then a cluster automorphism maps a labeled seed $\S=(\x, B)$ to a labeled seed $\S'=(\x', B')$.
Under our assumption that $B$ is indecomposable, we have the following
\begin{lem} \cite{ASS12}\label{lem:equivalent discription of clus-auto}
A $\Z$-algebra automorphism $f:\A\to\A$ is a cluster automorphism if and only if there exists a labeled seed $\S=(\x, B)$ of $\A$, such that $f(\x)$ is the cluster in a labeled seed $\S'=(\x', B')$ of $\A$ with $B'= B$ or $B'=-B$.
\end{lem}
We call the cluster automorphism such that $B= B'$ ($B= -B'$ respectively) a \emph{direct cluster automorphism} (an \emph{inverse cluster automorphism} respectively). Clearly, all the cluster automorphism of a cluster algebra $\A$ form a group with homomorphism composition as multiplication.
We call this group the \emph{cluster automorphism group} of $\A$, and denote it by $Aut(\A)$. We call the group $Aut^{+}(\A)$ consisting of the direct cluster automorphisms of $\A$ the \emph{direct cluster automorphism group} of $\A$, which is a subgroup of $Aut(\A)$ with index at most two, see \cite{ASS12}.\\

\begin{defn}(Automorphism of exchange graphs) \cite{S14,CZ15}
An automorphism of the exchange graph $E_\A$ of a cluster algebra $\A$ is an automorphism of $E_\A$ as a graph, that is, a permutation $\sigma$ of the vertex set, such that the pair of
vertices $(u,v)$ forms an edge if and only if the pair $(\sigma(u),\sigma(v))$ also forms an edge.
\end{defn}

Clearly, the natural composition of two automorphisms of $E_\A$ is again an automorphism. We define an \emph{automorphism group $Aut(E_\A)$} of $E_\A$ as a group consisting of automorphisms of $E_\A$.
It is clearly that a cluster automorphism induces a unique automorphism of the exchange graph. Thus $Aut(\A)$ is a subgroup of $Aut(E_\A)$, see \cite{CZ15}. By the definition, an automorphism $\sigma$ of an exchange graph maps clusters to clusters, and induces an automorphism of its dual graph: cluster complex $\Delta$, we denote this automorphism by $\sigma_\Delta$. Then $\sigma_\Delta$ is a permutation of cluster variables in $\X$, which maps a maximal simplex to a maximal simplex, but the map may not be compatible with the algebra relations among cluster variables in $\A$, thus it is not necessarily a cluster automorphism. In fact, $Aut(\A)$ maybe a proper subgroup of $Aut(E_\A)$, see Example \ref{exm:rank 2 case},\ref{exm:infinite rank 2 case2}. The following Lemma can be viewed as a description of $Aut(\A)$ as a subgroup of $Aut(E_\A)$, as those exchange graph automorphisms which happen to preserve B-matrices (perhaps up to global reversal of sign) up to simultaneously relabeling of the rows and columns. In this point of view, the main thrust of this paper is to show that, typically for the cluster algebras we consider, any graph automorphism has the property of preserving B-matrices.
\begin{lem}\label{lem:key lemma}
Let $\Phi: E_\A\to E_\A$ be an automorphism which maps a seed $\S=(\x,B)$ to a seed $\S'=(\x',B')$. If $B\cong B'$ or $B\cong-B'$ under the correspondence $\x\to \x'$, then the map $\x\to \x'$ induces a cluster automorphism $\Psi$ of $\A$ and the induced automorphism $\Psi_E: E_\A\to E_\A$ coincides with $\Phi$.
\end{lem}
\begin{proof}
Since $B\cong B'$ or $B\cong-B'$, the map $\x\to \x'$ induces a cluster automorphism $\Psi$ of $\A$ by Lemma \ref{lem:equivalent discription of clus-auto}. Noticing that $\Phi(\x)=\Psi(\x)$, then by inductions on the mutations, we have $\Phi=\Psi$ on each clusters of $E_\A$, so $\Phi=\Psi_E$ as automorphisms of the exchange graph $E_\A$.
\end{proof}

\section{Automorphism groups of exchange graphs}
\label{Sec exchange graph}
In this section we consider relations between the groups $Aut(\A)$ and $Aut(E_\A)$ for a cluster algebra $\A$ of finite type or of skew-symmetric finite mutation type. For this, it is needed to describe $E_\A$ more precisely. In the following we will recall the basic structures of $E_\A$ from \cite{FZ02,FZ03}, and then introduce layers of geodesic loops on $E_\A$.

\subsection{Layers of geodesic loops}
\label{Layers of geodesic loops}


Let $\S=(\x,B)$ be a labeled seed on the cluster pattern of $\A$.
Let $\x'$ be a proper subset of $\x$, then $\x'$ is a non-maximal simplex in the cluster complex $\Delta$. We denote by $\Delta_{\x'}$ the {\it link} of $\x \setminus \x'$, which is the simplicial complex on the ground set $\X_{\x'}=\{\alpha\in \X-(\x \setminus \x') : (\x \setminus \x')\cup \{\alpha\}\in \Delta\}$, such that $\x''$ is a simplex in $\Delta_{\x'}$ if and only if $\x \setminus \x'\cup \x''$ is a simplex in $\Delta$. Let $\Gamma_{\x'}$ be the dual graph of $\Delta_{\x'}$. We view $\Gamma_{\x'}$ as a subgraph of $E_\A$ whose vertices are the maximal simplices in $\Delta$ that contain ${\x \setminus \x'}$. In fact, as we explain now, $\Gamma_{\x'}$ is the exchange graph of a cluster algebra $\A_f$ defined by a frozen seed $\S_f=(\x',\x\setminus\x',B_f)$, which is the freezing of $\S$ at $\x \setminus \x'$ (see \cite[Definition 2.25]{CZ14}), where $B_f$ is obtained from $B$ by deleting the columns labeled by variables in $\x \setminus \x'$. Then elements in $\x \setminus \x'$ are coefficients of $\A_f$ (we refer to \cite{FZ02,FZ07} for a cluster algebra with coefficients). Let $\A'$ be cluster algebra defined by a seed $\S'=(\x',B')$, where $B'$ is obtained from $B$ by deleting rows and columns labeled by variables in $\x \setminus \x'$. In our settings, that is, cluster algebras are of finite type or of skew-symmetric finite type, the exchange graph of a cluster algebra (with coefficients) only depends on the principal part of the exchange matrix (see \cite{FZ03,CKLP13}) which is the submatrix labeled by $\x\setminus \x'\times \x\setminus \x'$, thus the graph $\Gamma_{\x'}$ coincides with the exchange graph $E_{\A'}$.\\

For an $2$-dimensional subcomplex ${\x'}$ of $\Delta$, we call the dual graph $\Gamma_{\x'}$ a {\it geodesic loop} of $E_\A$. We mention that the definition of geodesic loop is slightly different with the definition used in \cite{FZ03}, where a line is not a geodesic loop. If $\A$ is of finite type, then $E_\A$ is a finite graph, and $\Gamma_{\x'}$ is a polygon. Notice that in the seed $\S'=(\x',B')$ constructed above, $B'$ is of Dynkin type, that is, one of types $A_2,B_2,C_2$ or $G_2$. Therefore $\Gamma_{\x'}$ is a $h+2$-polygon, where $h$ is the Coxeter number of the corresponding Dynkin type, see \cite{FZ03}. If $\A$ is of finite mutation type, then $\Gamma_{\x'}$ may be a line. We fix a basepoint $\S=(\x,B)$ and introduce the following concept.\\


\begin{defn}
\begin{enumerate}
\item Let $\S'$ be a point of $E_\A$, the \emph{distance} $\ell(\S,\S')$ between $\S$ and $\S'$ is the minimal length of paths between $\S$ and $\S'$;
\item Let $L$ be a geodesic loop of $E_\A$, the \emph{distance} $\ell_\S(L)$ between $\S$ and $L$ is the minimal length $min\{\ell(\S,\S'),\S'\in L\}$;
\item Let $m\in\Z_{\geqslant 0}$ be a non-negative integer, denote by $\ell^m_\S$ the set of geodesic loop whose distance to $\S$ is $m$. We call it the \emph{$m$-layer} of geodesic loops of $E_\A$ based on $\S$;
\item For any $m\in\Z_{\geqslant 0}$, denote $N(\ell^m_\S)$ the set of amounts of edges belonging to geodesic loops in the \emph{$m$-layer} $\ell^m_\S$.
\end{enumerate}
\end{defn}
\begin{rem}\label{rem: layers of geodesic loops}
The following observations are directly derived from the definitions:
\begin{enumerate}
\item The elements in $\ell^0_\S$ are those geodesic loops $\Gamma_{\x'}$ for the $2$-dimensional subcomplex $\x'$ of $\Delta$, where $\x'$ is a subset of the cluster $\x$ in $\S$;
\item For $m_1\neq m_2$, $\ell^{m_1}_\S\cap\ell^{m_2}_\S=\emptyset$;
\item The disjoint union $\sqcup_{m\geqslant 0}\ell^m_\S$ is the set of all the geodesic loops of $E_\A$;
\item If $\sigma:E_\A\to E_{\A'}$ is an isomorphism of graphs, such that the image of $\S$ is $\S'$, then for every $m\in \Z_{\geqslant 0}$, $N(\ell^m_\S)=N(\ell^m_{\S'})$ as sets.
\end{enumerate}
\end{rem}

\subsection{Cases of rank two and rank three}
In this subsection, we consider the relations between $Aut(\A)$ and $Aut(E_\A)$ for a cluster algebra $\A$ of rank two or rank three.

\begin{exm}\label{exm:rank 2 case}
For a finite type cluster algebra $\A$ of rank $2$, that is, one of types $A_2, B_2, C_2$ or $G_2$, its exchange graph $E_\A$ is a $(h+2)$-polygon, thus $Aut(E_\A)$ is isomorphic to the dihedral group $\mathbb{D}_{h+2}$, where $h$ is the Coxeter number. If $\A$ is of type $A_2$, then $Aut(\A)\cong \mathbb{D}_{5}$ \cite{ASS12}, thus $Aut(\A)\cong Aut(E_\A)$. If $\A$ is of type $B_2, C_2$ or $G_2$, \cite[Theorem 3.5]{CZb15} shows that $Aut(\A)\cong \mathbb{D}_{(h+2)/2}$, thus $Aut(\A)\subsetneqq Aut(E_\A)$.
\end{exm}

\begin{exm}\label{exm:infinite rank 2 case}
For an infinite type skew-symmetric cluster algebra $\A$ of rank $2$, its exchange graph $E_\A$ is a line, thus $Aut(E_\A)=<s>\times<r>\cong \Z\rtimes\Z_2=\mathbb{D}_\infty$, where $s$ is a left shift of $E_\A$ which maps a cluster to the left adjacent cluster and $r$ is a reflection with respect to a fixed cluster. Then $s$ corresponds to a direct cluster automorphism of $\A$ and $r$ corresponds to an inverse cluster automorphism of $\A$, thus by Lemma \ref{lem:key lemma} $Aut(E_\A) \subseteq Aut(\A)$. Therefore $Aut(E_\A) \cong Aut(\A) \cong \mathbb{D}_\infty$.
\end{exm}

\begin{exm}\label{exm:infinite rank 2 case2}
For an infinite type non-skew-symmetric cluster algebra $\A$ of rank $2$, its exchange graph $E_\A$ is also a line, thus as showed in Example \ref{exm:infinite rank 2 case}, $Aut(E_\A)=<s>\times<r>\cong \Z\rtimes\Z_2$, where $s$ corresponds to a direct cluster automorphism of $\A$, while $r$ dose not correspond to any cluster automorphism of $\A$, since there is no non-trivial symmetry of the quiver in any seed of $\A$. Thus $Aut(\A) \cong \Z \subsetneqq Aut(E_\A)$.
\end{exm}

\begin{exm}\label{exm:A3 type exchange graph}
We consider the cluster algebra $\A$ of type $A_3$ with an initial labeled seed $\S_0=(\{x_1,x_2,x_3\}, Q)$, where $Q$ is
$\xymatrix{1\ar[r]&2&3\ar[l]}.$ Then its exchange graph $E_{\A}$ is depicted in Figure \ref{automorphisms of exchange graph of type $A_3$}.
Note that there are three quadrilaterals and six pentagons in $E_{\A}$. Then as shown in \cite[Example 3]{CZb15}, $Aut(\A)=<f_-,f_+>\cong \mathbb{D}_6$, where $f_-$ is defined by:
\begin{equation}
f_-: \begin{cases}
x_1 \mapsto x_1 \\ 
x_2 \mapsto \mu_2(x_2) \\ 
x_3 \mapsto x_3
\end{cases}
\end{equation}
It maps $\S_0$ to $\S_1$, and induces a reflection with respect to the horizontal central axis of $E_\A$. The cluster automorphism $f_+$ is defined by:
\begin{equation}
f_+: \begin{cases}
x_1 \mapsto \mu_1(x_1) \\ 
x_2 \mapsto x_2 \\ 
x_3 \mapsto \mu_3(x_3)
\end{cases}
\end{equation}
It gives a reflection on $E_\A$, which maps $\S_0$ to $\S_5$.
In fact as shown in \cite{CZb15}, a direct cluster automorphism of $\A$ is of the form $(f_+f_-)^m, 0 \leq m \leq 5$, which induces a rotation of seeds in $\{\S_0, \S_1, \S_2, \S_3, \S_4, \S_5\}$, thus $Aut^+(\A)$ can be viewed as the symmetry group of the \emph{bipartite belt} consisting of seeds in $\{\S_0, \S_1, \S_2, \S_3, \S_4, \S_5\}$, where the quivers in these seeds are the bipartite quivers isomorphic to $Q$.\\

We will prove that any automorphisms of $E_\A$ is induced from an element in $Aut(\A)$, and thus $Aut(\A)\cong Aut(E_\A)$. For this purpose, we should show the following claims:

(1) there exists no automorphism of $E_\A$ which maps $\S_0$ to a vertex excepting for $\S_i, 0 \leq i \leq 5$;

(2) if an automorphism of $E_\A$ maps $\S_0$ to $\S_i$, $0\leq i \leq 5$, then it is induced from a cluster automorphism of $\A$.

Let $\sigma$ be an automorphism of $E_\A$, due to symmetries of $E_\A$, we only show that $\sigma(\S)\neq O_i$, $i=1,2,3$. By a direct computation, the sets of numbers for the layers of geodesic loops based on these vertices are as follows:
\begin{align*}
N(\ell^0_\S)=\{4,5,5\},~N(\ell^1_\S)=\{4,5,5\},~N(\ell^2_\S)=\{5,5\},~N(\ell^3_\S)=\{4\};~~~~~\\
N(\ell^0_{O_1})=\{4,5,5\},~N(\ell^1_{O_1})=\{5,5,5\},~N(\ell^2_{O_1})=\{4,4\},~N(\ell^3_{O_1})=\{5\};\\
N(\ell^0_{O_2})=\{5,5,5\},~N(\ell^1_{O_1})=\{4,4,4\},~N(\ell^2_{O_1})=\{5,5,5\};~~~~~~~~~~~~~~~~~~\\
N(\ell^0_{O_3})=\{4,5,5\},~N(\ell^1_{O_3})=\{5,5,5\},~N(\ell^2_{O_3})=\{4,4\},~N(\ell^3_{O_3})=\{5\};\\
\end{align*}
Then by Remark \ref{rem: layers of geodesic loops} (4), $\sigma(\S_0)\neq O_i$, $i=1,2,3$. So the claim one is affirmed.\\

Now we consider the claim two. Still due to the symmetries of the graph, we may assume that $\sigma(\S_0) = \S_0$. Since $\sigma$ is a graph automorphism, it can be seen that there are two possibilities of $\sigma$, one is the identity, another is the reflection $f_0$ with respect to the vertical central axis of $E_\A$, as depicted in Figure \ref{automorphisms of exchange graph of type $A_3$}. Note that the identity graph automorphism is induced from the identity automorphism of the cluster algebra, while the graph automorphism $f_0$ is induced from the cluster automorphism $(f_+f_-)^3$ by a direct computation.
Therefore the claim two is true and we have $Aut(E_{\A})\cong Aut(\A)\cong \mathbb{D}_6$.

\begin{figure}
\begin{center}
{\begin{tikzpicture}[scale=0.8]

\node[] (C) at (0,4)
						{$\S_4$};
\node[] (C) at (0,2.8)
						{$\S_5$};
\node[] (C) at (0,0.6)
						{$\S_0$};
\node[] (C) at (3,0)
						{$O_3$};
\node[] (C) at (1.8,0)
						{$O_2$};
\node[] (C) at (0.65,1.7)
						{$O_1$};
\draw[thick] (0,3.7) -- (0,3.1);
\draw[thick] (2.1,0) -- (2.7,0);
\draw[thick] (0.3,3.7) -- (2.85,0.2);
\draw[thick] (0.1,2.55) -- (0.45,2.0);
\draw[thick] (0.2,0.9) -- (0.45,1.4);
\draw[thick] (1.65,0.2) -- (0.8,1.4);

\node[] (C) at (-3,0)
						{$\bullet$};
\node[] (C) at (-1.8,0)
						{$\bullet$};
\node[] (C) at (-0.65,1.7)
						{$\bullet$};
\draw[thick] (-2.1,0) -- (-2.7,0);
\draw[thick] (-0.3,3.7) -- (-2.85,0.2);
\draw[thick] (-0.1,2.55) -- (-0.45,2.0);
\draw[thick] (-0.2,0.9) -- (-0.45,1.4);
\draw[thick] (-1.65,0.2) -- (-0.8,1.4);


\node[] (C) at (0,-4)
						{$\S_3$};
\node[] (C) at (0,-2.8)
						{$\S_2$};
\node[] (C) at (0,-0.6)
						{$\S_1$};

\node[] (C) at (0.65,-1.7)
						{$\bullet$};
\draw[thick] (0,-3.7) -- (0,-3.1);
\draw[thick] (2.1,0) -- (2.7,0);
\draw[thick] (0.3,-3.7) -- (2.85,-0.2);
\draw[thick] (0.1,-2.55) -- (0.45,-2.0);
\draw[thick] (0.2,-0.9) -- (0.45,-1.4);
\draw[thick] (1.65,-0.2) -- (0.8,-1.4);

\node[] (C) at (-3,0)
						{$\bullet$};
\node[] (C) at (-1.8,0)
						{$\bullet$};
\node[] (C) at (-0.65,-1.7)
						{$\bullet$};
\node[] (C) at (0,-5.3)
						{$f_0$};
\node[] (C) at (4,0)
						{$f_-$};
\node[] (C) at (-4,1.6)
						{$f_+$};
\draw[thick] (-2.1,0) -- (-2.7,0);
\draw[thick] (-0.3,-3.7) -- (-2.85,-0.2);
\draw[thick] (-0.1,-2.55) -- (-0.45,-2.0);
\draw[thick] (-0.2,-0.9) -- (-0.45,-1.4);
\draw[thick] (-1.65,-0.2) -- (-0.8,-1.4);

\draw[thick] (0,-0.3) -- (0,0.3);
\draw[thick,dashed,<->] (-0.5,-4.7) -- (0.5,-4.7);
\draw[thick,dashed,<->] (3.5,0.5) -- (3.5,-0.5);
\draw[thick,dashed,<->] (-3.5,1) -- (-3.5,2);
\end{tikzpicture}}
\end{center}
\begin{center}
\caption{The exchange graph of a cluster algebra of type $A_3$}
\label{automorphisms of exchange graph of type $A_3$}
\end{center}
\end{figure}
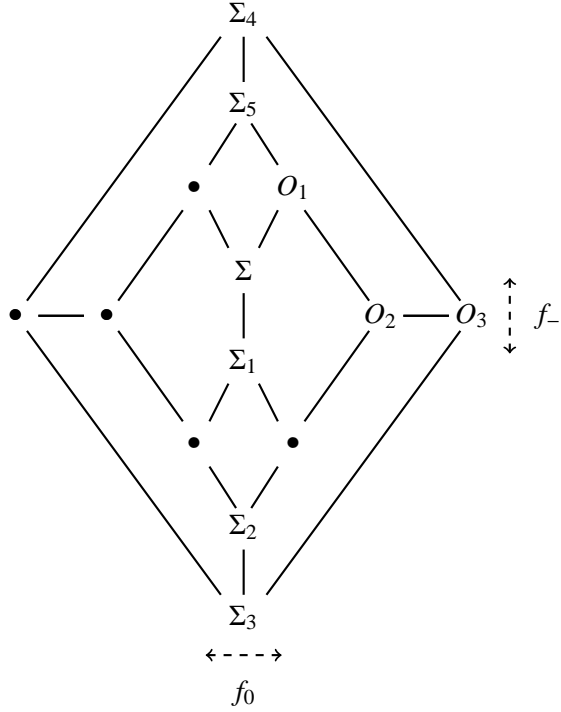
\end{exm}
\begin{exm}\label{exm:B3 C3 type exchange graph}
It is known from a result in \cite{FZ032} that the cluster algebras of type $B_n$ and type $C_n$ have the same exchange graph. Based on a seed $\S_0$,
the exchange graph of a cluster algebra $\A$ of type $B_3$ or type $C_3$
is depicted in Figure \ref{automorphisms of exchange graph of type $BC_3$}.
For the cluster algebra of type $B_3$, the quiver of the initial seed $\S_0$ is \begin{center}
{\begin{tikzpicture}
\node[] (C) at (-1.5,0)  {$1$};
\node[] (C) at (0,0)  {$2$};	
\node[] (C) at (1.5,0)  {$3.$};
\node[] (C) at (0.7,0.3)  {\qihao{(2,1)}};
\draw[<-,thick] (-0.2,0) -- (-1.3,0);
\draw[<-,thick] (0.2,0) -- (1.3,0);
\end{tikzpicture}}
\end{center}
For the cluster algebra of type $C_3$, the quiver of the initial seed $\S_0$ is
\begin{center}
{\begin{tikzpicture}
\node[] (C) at (-1.5,0)  {$1$};
\node[] (C) at (0,0)  {$2$};						
\node[] (C) at (1.5,0)  {$3.$};
\node[] (C) at (0.7,0.3)  {\qihao{(1,2)}};
\draw[<-,thick] (-0.2,0) -- (-1.3,0);
\draw[<-,thick] (0.2,0) -- (1.3,0);
\end{tikzpicture}}
\end{center}
Let $\sigma$ be an automorphism of $E_\A$. As showed by Example 4 in \cite{CZb15}, we have $Aut(\A)\cong \mathbb{D}_4$ and $\{\S_0, \S_1, \S_2, \S_3, \S_4, \S_5, \S_6, \S_7\}$ are all the seeds whose quivers are isomorphic to $Q$. Similarly, to get $Aut(E_\A)\cong Aut(\A)$, we should prove the two claims stated in Example \ref{exm:A3 type exchange graph}. For the first claim, we only need to prove that $\sigma(\S_0)\neq O_{i}(i=1,2,3,4)$ in Figure \ref{automorphisms of exchange graph of type $BC_3$}, and this can be obtained by the fact that these seeds have different combinatorial numbers of layers of geodesic loops:
\begin{align*}
N(\ell^0_{\S_0})=\{4,5,6\},~N(\ell^1_{\S_0})=\{4,5,6\};~\\
N(\ell^0_{O_1})=\{5,6,6\};~~~~~~~~~~~~~~~~~~~~~~~~~~~~~~\\
N(\ell^0_{O_2})=\{4,5,6\},~N(\ell^1_{O_2})=\{5,6,6\};\\
N(\ell^0_{O_3})=\{4,5,6\},~N(\ell^1_{O_3})=\{5,6,6\};\\
N(\ell^0_{O_4})=\{5,6,6\}.~~~~~~~~~~~~~~~~~~~~~~~~~~~~~~~\\
\end{align*}
For the second claim, we may also assume $\sigma(\S_0)= \S_0$. Since $N(\ell^0_{\S_0})=\{4,5,6\}$, there are neither rotation symmetries nor reflection symmetries of $E_\A$ at $\S_0$. So $\sigma$ must be the identity automorphism of $E_\A$, which is induced from the identity automorphism of the cluster algebra. Noticing that there are eight elements in $Aut(\A)\cong \mathbb{D}_4$, where each one corresponds a graph automorphism which maps $\S_0$ to $\S_i, 0\leq i\leq7$.
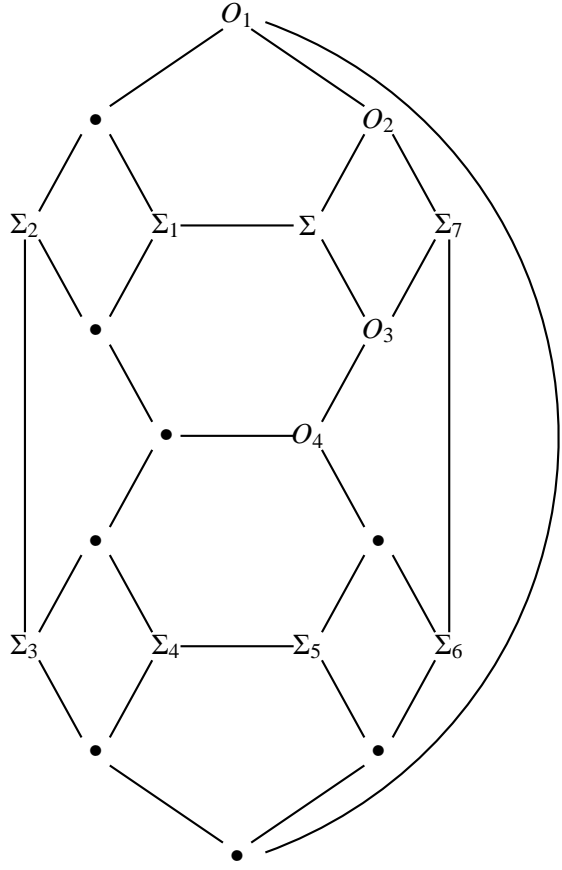
\begin{figure}
\begin{center}
{\begin{tikzpicture}[scale=0.8]

\node[] (C) at (0,6)
						{$O_1$};

\node[] (C) at (1,0)
						{$O_4$};
\node[] (C) at (2,1.5)
						{$O_3$};
\node[] (C) at (1,3)
						{$\S$};
\node[] (C) at (3,3)
						{$\S_7$};
\node[] (C) at (2,4.5)
						{$O_2$};

\draw[thick] (-0.8,0) -- (0.8,0);
\draw[thick] (-0.8,3) -- (0.8,3);

\draw[thick] (1.2,0.2) -- (1.8,1.3);
\draw[thick] (1.2,2.8) -- (1.8,1.7);
\draw[thick] (1.2,3.2) -- (1.8,4.3);
\draw[thick] (2.8,3.2) -- (2.2,4.3);
\draw[thick] (1.8,4.7) -- (0.2,5.8);
\draw[thick] (2.2,1.7) -- (2.8,2.8);

\node[] (C) at (-1,0)
						{$\bullet$};
\node[] (C) at (-2,1.5)
						{$\bullet$};
\node[] (C) at (-1,3)
						{$\S_1$};
\node[] (C) at (-3,3)
						{$\S_2$};
\node[] (C) at (-2,4.5)
						{$\bullet$};

\draw[thick] (-1.2,0.2) -- (-1.8,1.3);
\draw[thick] (-1.2,2.8) -- (-1.8,1.7);
\draw[thick] (-1.2,3.2) -- (-1.8,4.3);
\draw[thick] (-2.8,3.2) -- (-2.2,4.3);
\draw[thick] (-1.8,4.7) -- (-0.2,5.8);
\draw[thick] (-2.2,1.7) -- (-2.8,2.8);
\node[] (C) at (0,-6)
						{$\bullet$};
\node[] (C) at (2,-1.5)
						{$\bullet$};
\node[] (C) at (1,-3)
						{$\S_5$};
\node[] (C) at (3,-3)
						{$\S_6$};
\node[] (C) at (2,-4.5)
						{$\bullet$};

\draw[thick] (-0.8,-3) -- (0.8,-3);

\draw[thick] (1.2,-0.2) -- (1.8,-1.3);
\draw[thick] (1.2,-2.8) -- (1.8,-1.7);
\draw[thick] (1.2,-3.2) -- (1.8,-4.3);
\draw[thick] (2.8,-3.2) -- (2.2,-4.3);
\draw[thick] (1.8,-4.7) -- (0.2,-5.8);
\draw[thick] (2.2,-1.7) -- (2.8,-2.8);

\node[] (C) at (-2,-1.5)
						{$\bullet$};
\node[] (C) at (-1,-3)
						{$\S_4$};
\node[] (C) at (-3,-3)
						{$\S_3$};
\node[] (C) at (-2,-4.5)
						{$\bullet$};

\draw[thick] (-1.2,-0.2) -- (-1.8,-1.3);
\draw[thick] (-1.2,-2.8) -- (-1.8,-1.7);
\draw[thick] (-1.2,-3.2) -- (-1.8,-4.3);
\draw[thick] (-2.8,-3.2) -- (-2.2,-4.3);
\draw[thick] (-1.8,-4.7) -- (-0.2,-5.8);
\draw[thick] (-2.2,-1.7) -- (-2.8,-2.8);
\draw[thick] (3,2.8) -- (3,-2.8);
\draw[thick] (-3,2.8) -- (-3,-2.8);
\draw[thick] (0.4,5.9) arc (70:-70:6.3);

\end{tikzpicture}}
\end{center}
\begin{center}
\caption{The exchange graph of a cluster algebra of type $B_3$ or type $C_3$}
\label{automorphisms of exchange graph of type $BC_3$}
\end{center}
\end{figure}
\end{exm}

\begin{exm}\label{exm:F_4 auto gp of exchange graph}
For cluster algebra of type $F_4$, let the quiver $Q$ of a seed $\S$ be
\begin{center}
{\begin{tikzpicture}
\node[] (C) at (-1.5,0)  {$1$};
\node[] (C) at (0,0)  {$2$};						
\node[] (C) at (1.5,0)  {$3$};
\node[] (C) at (3,0)  {$4.$};
\node[] (C) at (0.8,0.3)  {\qihao{(2,1)}};
\draw[<-,thick] (-0.2,0) -- (-1.3,0);
\draw[<-,thick] (0.2,0) -- (1.3,0);
\draw[->,thick] (1.7,0) -- (2.8,0);
\end{tikzpicture}}
\end{center}
Then $Aut(\A)\cong \mathbb{D}_7$ \cite{CZb15}.
The variables $x_1,x_2$,$x_3$ and the corresponding full subquiver of $Q$ form a seed $\S_1$ of type $B_3$, while $x_2,x_3$,$x_4$ and the corresponding full subquiver of $Q$ form a seed $\S_2$ of type $C_3$. By pinning down $\S$, rotating the graph $E_\A$ induces an automorphism $\sigma$ of $E_\A$, which exchanges the graph $E_{\A_{\S_1}}$ and the graph $E_{\A_{\S_2}}$. However $\sigma$ does not induces a cluster automorphism of $\A$, and $Aut(\A)\cong \mathbb{D}_7\subsetneqq \mathbb{D}_7\rtimes \Z_2 \cong Aut(E_\A)$.
\end{exm}

\begin{prop}\label{prop:3 points Dynkin type algebra}
Let $Q$ be a connected quiver with three vertices, which is finite type. Let $\S=(\x,Q)$ and $\S'=(\x',Q')$ be two seeds (not necessarily mutation equivalent with each other). If there is an isomorphism $\sigma:E_\A\to E_{\A'}$ such that $\sigma(\S)=\S'$, then $\S'=(\x',Q')$ is a finite type seed with $Q'$ connected, and
\begin{enumerate}
\item if $\S$ is of type $A_3$, then $Q'\cong Q$ (or $Q^{op}$);
\item if $\S$ is of type $B_3$ and $\S'$ is not of type $C_3$, then $Q'\cong Q$ (or $Q^{op}$);
\item if $\S$ is of type $C_3$ and $\S'$ is not of type $B_3$, then $Q'\cong Q$ (or $Q^{op}$).
\end{enumerate}
\end{prop}
\begin{proof}
Clearly, since $E_\A'\cong E_\A$ is of finite, $Q'$ is a Dynkin type quiver with three vertices. If $Q$ is of type $A_3$, then by Example \ref{exm:A3 type exchange graph}, $$N(\ell^0_\S)=\{4,5,5\}~~\textrm{or}~~\{5,5,5\}.$$
If $Q$ is of type  $B_3$ (or $C_3$), then from Example \ref{exm:B3 C3 type exchange graph},
$$N(\ell^0_\S)=\{4,5,6\}~~\textrm{or}~~\{5,6,6\}.$$
If $Q'$ is a union of a quiver of type $A_2$ and a point, then from Example \ref{exm:rank 2 case},
$$N(\ell^0_\S)=\{4,4,5\}.$$
If $Q'$ is a union of a quiver of type $B_2$ (or $C_2$) and a point, then from Example \ref{exm:rank 2 case},
$$N(\ell^0_\S)=\{4,4,6\}.$$
If $Q'$ is a union of a quiver of type $G_2$ and a point, then from Example \ref{exm:rank 2 case},
$$N(\ell^0_\S)=\{4,4,8\}.$$
Thus we get the proof by Remark \ref{rem: layers of geodesic loops}.
\end{proof}

\begin{exm}\label{exm:infinite rank 3 case}
Let $Q$ be the quiver in Figure \ref{quiver of type $tA2$}, we call it of type $\tilde{A}_2$. Then it is not hard to see that if a quiver in the mutation class of $Q$ is not isomorphic to $Q$, then it must be isomorphic to the quiver $Q'$ in Figure \ref{quiver of type $tA2$}. Let $\A$ be a cluster algebra with an initial seed $\S=(\{x_1,x_2,x_3\},Q)$, similarly to above examples, to show that $Aut(\A)\cong Aut(E_\A)$, we only need to notice that:
\begin{align*}
~N(\ell^0_\S)=\{5,5,\infty\},\\
N(\ell^0_{\S'})=\{5,5,5\},
\end{align*}
where $\S'$ is a seed of $\A$ with quiver isomorphic to $Q'$.
In fact, from section 3.3 in \cite{ASS12}, $Aut(\A)=<r_1,r_2 | r_1r_2=r_2r_1, {r_1}^{2}=r_2>\rtimes<\sigma | \sigma^{2}=1>\cong \mathbb{H}_{2,1}\rtimes \Z_2$, where
\begin{equation}
r_1: \begin{cases}
x_1 \mapsto x_3 \\
x_2 \mapsto \mu_1(x_1) \\
x_3 \mapsto x_2
\end{cases}
\end{equation}

\begin{equation}
r_2: \begin{cases}
x_1 \mapsto x_2 \\
x_2 \mapsto \mu_3\mu_1(x_3) \\
x_3 \mapsto \mu_1(x_1)
\end{cases}
\end{equation}

\begin{equation}
\sigma: \begin{cases}
x_1 \mapsto x_2 \\
x_2 \mapsto x_1 \\
x_3 \mapsto x_3
\end{cases}
\end{equation}
Thus $Aut(E_\A)\cong \mathbb{H}_{2,1}\rtimes \Z_2$

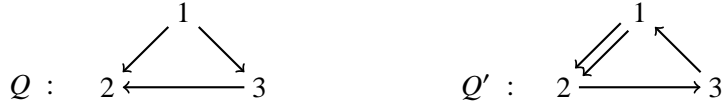
\begin{figure}
\begin{center}
{\begin{tikzpicture}

\node[] (C) at (-4,-1)
						{$Q~:$};

\node[] (C) at (-2,0)
						{$1$};

\node[] (C) at (-3,-1)
						{$2$};
\node[] (C) at (-1,-1)
						{$3$};

\draw[->,thick] (-2.2,-0.2) -- (-2.8,-0.8);
\draw[->,thick] (-1.8,-0.2) -- (-1.2,-0.8);
\draw[<-,thick] (-2.8,-1) -- (-1.2,-1);

\node[] (C) at (-4,-3)
						{$Q'~:$};
\node[] (C) at (-2,-2)
						{$1$};

\node[] (C) at (-3,-3)
						{$2$};
\node[] (C) at (-1,-3)
						{$3$};

\draw[->,thick] (-2.25,-2.15) -- (-2.85,-2.75);
\draw[->,thick] (-2.15,-2.25) -- (-2.75,-2.85);
\draw[<-,thick] (-1.8,-2.2) -- (-1.2,-2.8);
\draw[->,thick] (-2.8,-3) -- (-1.2,-3);
\end{tikzpicture}}
\end{center}
\begin{center}
\caption{quivers of type $\tilde{A}_2$}
\label{quiver of type $tA2$}
\end{center}
\end{figure}
\end{exm}

\begin{exm}\label{exm:infinite rank 3 case 2}
Let $\A$ be a cluster algebra from an once punctured torus, we call it a cluster algebra of type $T_3$, then it is of finite mutation type with quiver always isomorphic to the quiver in Figure \ref{quiver of type $T3$}. Then by Lemma \ref{lem:key lemma}, we have $Aut(\A)\cong Aut(E_\A)$.
\begin{figure}
\begin{center}
{\begin{tikzpicture}

\node[] (C) at (0,0)
						{$1$};

\node[] (C) at (-1,-1)
						{$2$};
\node[] (C) at (1,-1)
						{$3$};

\draw[->,thick] (-0.25,-0.15) -- (-0.85,-0.75);
\draw[->,thick] (-0.15,-0.25) -- (-0.75,-0.85);
\draw[<-,thick] (0.15,-0.25) -- (0.75,-0.85);
\draw[<-,thick] (0.25,-0.15) -- (0.85,-0.75);
\draw[->,thick] (-0.8,-1.07) -- (0.8,-1.07);
\draw[->,thick] (-0.8,-0.93) -- (0.8,-0.93);
\end{tikzpicture}}
\end{center}
\begin{center}
\caption{quiver of type $T_3$}
\label{quiver of type $T3$}
\end{center}
\end{figure}
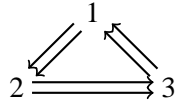
\end{exm}

\begin{cor}
Let $\A$ and $\A'$ be two cluster algebras of finite type, or of skew-symmetric finite mutation type, with rank equal to $2$ or $3$. Let $\S=(\x,B)$ and $\S'=(\x',B')$ be two seeds of $\A$ and $\A'$ respectively.
If $N(\ell^k_\S)=N(\ell^k_{\S'})$ for any $k\in \mathbb{Z}_{\geq 0}$, then there exists an isomorphism $\Phi: E_\A\to E_{\A'}$ such that $\Phi(\x)=\x'$.
\end{cor}
\begin{proof}
This follows from Example \ref{exm:A3 type exchange graph}, Example \ref{exm:B3 C3 type exchange graph}, Example \ref{exm:infinite rank 3 case} and Example \ref{exm:infinite rank 3 case 2}
\end{proof}
We expect the result in the Corollary is true for any finite type cluster algebras and finite mutation type cluster algebras. It means that for any seed $\S$, the set $N(\ell^k_\S)$ characterizes the exchange graph.

\begin{lem}\label{lem: mutation finite infinite rank 3 types}
Let $Q$ be a connected skew-symmetric quiver of finite mutation type.
\begin{enumerate}
\item If there are $3$ vertices in $Q$, then $Q$ is one of the following types:
\begin{enumerate}
\item [(1)] $A_3$ type;
\item [(2)] $\tilde{A}_2$ type;
\item [(3)] $T_3$ type.
\end{enumerate}
\item If there are at least $4$ vertices in $Q$, then any full subquiver of $Q$ with three vertices is of type $A_3$ or of type $\tilde{A}_2$.
\end{enumerate}
\end{lem}
\begin{proof}
\begin{enumerate}
\item From the classification of cluster algebras of finite mutation type, $Q$ must be block-decomposable, then the proof is a straightforward check by gluing the blocks in \ref{fig:blocks-I-V}.
\item We only need to notice that a quiver of type $T_3$ is obtained by gluing two blocks of type II in Figure \ref{fig:blocks-I-V}, and thus one can not further glue it with a block to obtain a connected quiver of finite mutation type.
\end{enumerate}
\end{proof}
It is clear that if for any quiver in the mutation equivalent class of $Q$, the number of arrows between any two vertices is at most $2$, then $Q$ is of finite mutation type. The above lemma shows that the inverse statement is also true for the cases when there are at least 3
vertices, that is, we have the following corollary, which has been stated in \cite[Corollary 8]{DO08}.
\begin{cor}\label{cor:equ of finite mutation type}
A connected quiver $Q$ with at least 3 vertices is of finite mutation type if and only if for any quiver in its mutation class the number of arrows between any two vertices is at most $2$.
\end{cor}

\begin{prop}\label{prop: mutation finite infinite rank 3 types}
Let $Q$ be a connected skew-symmetric quiver with three vertices, which is of finite mutation type. Let $\S=(\x,Q)$ and $\S'=(\x',Q')$ be two seeds. If there is an isomorphism $\sigma:E_\A\to E_{\A'}$ such that $\sigma(\S)=\S'$, then $Q'\cong Q$ or  $Q'\cong Q^{op}$.
\end{prop}
\begin{proof}
Similar to Proposition \ref{prop:3 points Dynkin type algebra}, this follows from Lemma \ref{lem: mutation finite infinite rank 3 types}, Example \ref{exm:infinite rank 3 case 2}, Example \ref{exm:A3 type exchange graph} and Example \ref{exm:infinite rank 3 case}.
\end{proof}

\subsection{General cases}
\begin{thm}\label{thm:auto gp of finite type ex graph}
Let $\A$ be a cluster algebra of finite type. Assume that it is not of type $F_4$, let $\S=(\x,Q)$ be a labeled seed of $\A$, where $Q$ is a connected quiver with at least three vertices. Then we have $Aut(\A)\cong Aut(E_\A)$.
\end{thm}
\begin{proof}
We need to show that $Aut(E_\A)\subseteq Aut(\A)$. Let $\Phi$ be any automorphism of $E_\A$. Then it induces an automorphism $\phi$ of the complex $\Delta$, in particular, which gives a permutation on the cluster variable set $\X$.
Let $\x'\subseteq \x$ be a $3$-dimensional complex, and let $Q(\x')$ be the full subquiver of $Q(\x)$ with vertices indexed by the variables in $\x'$. Let $\A'$ be the cluster algebra defined by the seed $\S'=(\x',Q(\x'))$.

Noticing that since $\phi$ is an automorphism of a complex, it maps a simplex to a simplex, and thus induces a bijection from $\X_{\x'}=\{\alpha\in \X-(\x \setminus \x') : \x \setminus \x'\cup \{\alpha\}\in \Delta\}$ to $\X_{\phi(\x')}=\{\alpha\in \X-\phi(\x \setminus \x') : \phi(\x \setminus \x')\cup \{\alpha\}\in \Delta\}$, and also induces an isomorphism $\phi_{\x'}$ from the link $\Delta_{\x'}$ to the link $\Delta_{\phi(\x')}$. Moreover, the duality of the isomorphism $\phi_{\x'}$ gives an isomorphism between the dual graphs of the complexes, that is we have an isomorphism
\begin{equation}\label{eq:iso-on-dual-graphs}
\Phi_{\x'}: \Gamma_{\x'}\to \Gamma_{\phi(\x')}.
\end{equation}
Let $\overline{\S}'=(\phi(\x'),Q(\phi(\x')))$ be a seed, where $Q(\phi(\x'))$ is the full subquiver of $Q(\Phi(\x))$ whose vertices are those labeled by elements in $\phi(\x')$. Let $\overline{\A}'$ be the cluster algebra defined by $\overline{\S}'$. As showed in the beginning of subsection \ref{Layers of geodesic loops}, there are isomorphisms $\Gamma_{\x'}\cong E_{\A'}$ and $\Gamma_{\phi(\x')}\cong E_{\overline{\A}'}$. Then combining with the isomorphism \eqref{eq:iso-on-dual-graphs}, we have $E_{\A'}\cong E_{\overline{\A}'}$. Since $\A$ is not of type $F_4$, if $Q(\x')$ is of type $B_3$ (type $C_3$ respectively), then $Q(\phi(\x'))$ is not of type $C_3$ (type $B_3$ respectively). Thus by Proposition \ref{prop:3 points Dynkin type algebra}, $Q(\x'))\cong Q(\phi(\x'))$ or $Q(\x')\cong {Q(\phi(\x'))}^{op}$.

Let $\x'=\{x_1,x_2,x_3\}\subseteq \x$ and $\x''=\{x_2,x_3,x_4\}\subseteq \x$ be two 3-dimensional complexes with exactly two common elements.
By above discussion, we have $Q(\x'))\cong Q(\phi(\x'))$ or $Q(\x')\cong {Q(\phi(\x'))}^{op}$, and $Q(\x''))\cong Q(\phi(\x''))$ or $Q(\x'')\cong {Q(\phi(\x''))}^{op}$. Now assume $b_{x_2x_3}\neq 0$, that is there exists at least one arrow in $Q(\x)$ between the vertices labeled by $x_2$ and $x_3$, then simultaneously we have $Q(\x'))\cong Q(\phi(\x'))$ and $Q(\x''))\cong Q(\phi(\x''))$, or $Q(\x')\cong {Q(\phi(\x'))}^{op}$ and $Q(\x'')\cong {Q(\phi(\x''))}^{op}$. Finally, due to the arbitrariness of the choice of $\x'$ and the connectness of the quiver, one may show that $Q(\Phi(\x))\cong Q$ or $Q(\Phi(\x))\cong Q^{op}$.
See the inductive process in the following picture:

\[
{\begin{tikzpicture}
\node[] (C) at (-.5,0)
						{$\cdots$};
\node[] (C) at (0,0)
						{$x_{0}$};
\node[] (C) at (1,0)
						{$x_1$};
\node[] (C) at (2,0)
						{$x_2$};
\node[] (C) at (3,0)
						{$x_3$};
\node[] (C) at (4,0)
						{$x_4$};
\node[] (C) at (5,0)
						{$x_5$};
\node[] (C) at (6,0)
						{$x_6$};
\node[] (C) at (6.5,0)
						{$\cdots$};

\node[] (C) at (2,-.5)
						{$\x'$};
\node[] (C) at (3,.6)
						{$\x''$};
\node[] (C) at (4,-0.5)
						{$\x'''$};

\draw[dashed,thick] (0.2,0.05) -- (0.8,0.05);
\draw[dashed,thick] (1.2,0.05) -- (1.8,0.05);
\draw[dashed,thick] (2.2,0.05) -- (2.8,0.05);
\draw[dashed,thick] (3.2,0.05) -- (3.8,0.05);
\draw[dashed,thick] (4.2,0.05) -- (4.8,0.05);
\draw[dashed,thick] (5.2,0.05) -- (5.8,0.05);

\draw [
    thick,
    decoration={
        brace,
        mirror,
        raise=0.25cm
    },
    decorate
](1.05,0) -- (2.95,0);
\draw [
    thick,
    decoration={
        brace,
        mirror,
        raise=0.25cm
    },
    decorate
](3.05,0) -- (4.95,0);
\draw [
    thick,
    decoration={
        brace,
        raise=0.25cm
    },
    decorate
](2.05,0) -- (3.95,0);
\end{tikzpicture}}
\]

Therefore $\Phi:E_{\A}\to E_{\A}$ induces an cluster automorphism of $\A$ by Lemma \ref{lem:key lemma}. Thus $Aut(E_\A)\subseteq Aut(\A)$ and we have $Aut(E_\A)\cong Aut(\A)$.
\end{proof}

\begin{rem}\label{rem:auto-gp-of-finite-type}
By combining the above theorem, Table 1 in \cite{ASS12} and Theorem 3.5 in \cite{CZb15}, we may compute the automorphism groups of the exchange graphs of cluster algebras of finite type, see the table \ref{table:auto group of exchange graphs}. The cases of rank two and type $F_4$ is computed in Example \ref{exm:rank 2 case} and Example \ref{exm:F_4 auto gp of exchange graph} respectively.
\end{rem}

\begin{thm}\label{thm:auto gp of finite mutation type ex graph}
Let $\A$ be a connected skew-symmetric cluster algebra of finite mutation type, then $Aut(\A)\cong Aut(E_\A)$.
\end{thm}
\begin{proof}
If $\A$ is of finite type of rank $2$, that is, of type $A_2$, then the result follows from Example \ref{exm:rank 2 case}. If $\A$ is of infinite type of rank $2$, then the result follows from Example \ref{exm:infinite rank 2 case}.
When the rank of $\A$ is at least $3$, the proof is similar to the proof of Theorem \ref{thm:auto gp of finite type ex graph} by using the connectness of the cluster algebra and Proposition \ref{prop: mutation finite infinite rank 3 types}.
\end{proof}

\begin{cor}
Let $\A$ be a connected cluster algebra of finite type or of skew-symmetric finite mutation type, then an automorphism of $E_\A$ is determined by the image of any fixed seed $\S$ and the images of the seeds adjacent to $\S$. More precisely, let $\S=(\x,B)$ be a seed on $E_\A$, then an automorphism $\Phi: E_\A\to E_\A$ is determined by a pair $(\Sigma',\phi)$, where $\S'=(\x',B')$ is a seed on $E_\A$ and $\phi: \x\to \x'$ is a bijection such that $\Phi(\S)=\S'$ and $\Phi(\mu_{x}(\x))=\mu_{\phi(x)}(\x')$ for any $x\in \x$.
\end{cor}
\begin{proof}
If $\A$ is of finite type of rank $2$ and type of $F_4$, then the conclusion is clearly. Otherwise, note that a cluster automorphism is determined by such a pair $(\S',\phi)$, thus the proof follows from Theorem \ref{thm:auto gp of finite type ex graph} and Theorem \ref{thm:auto gp of finite mutation type ex graph}.
\end{proof}

\section*{Acknowledgements}
Both authors thank for anonymous reviewer's careful reading and many valuable suggestions. In particular, the reference \cite{DO08} is pointed out by the reviewer.

\bigskip\bigskip

{\small Wen Chang\\
School of Mathematics and Information Science, Shaanxi Normal University, Xi'an 710062, China \&\\
Department of Mathematical Sciences, Tsinghua University, Beijing 10084, China\\
Email: {\tt changwen161@163.com}

	\bigskip
Bin Zhu\\
Department of Mathematical Sciences, Tsinghua University, Beijing 10084, China\\
Email: {\tt zhu-b@mail.tsinghua.edu.cn}}

\end{document}